\documentclass[a4paper]{amsart}
\usepackage{amsfonts,amsmath,amssymb}
\usepackage[cp1251]{inputenc}
\usepackage[T2A]{fontenc}

\newtheorem{theorem}{Theorem}

\theoremstyle{remark}
\newtheorem*{remark}{Remark}
\theoremstyle{definition}

\def\C{\mathbb C}

\def\sK{\mathcal K}

\def\zk{\mathcal Z_{\mathcal K}}

\def\Tor{\mathop{\mathrm{Tor}}\nolimits}

\def\pt{\mathit{pt}}

\def\colim{\mathop{\mathrm{colim}}\nolimits}

\newcommand{\cat}[1]{\mbox{\sc #1}}
\begin{document}
\title{On the cohomology of quotients of moment-angle complexes}

\author{Taras Panov}
\address{Department of Mathematics and Mechanics, Moscow
State University
\newline\indent Institute for Theoretical and Experimental Physics,
Moscow, Russia,\quad \emph{and}
\newline\indent Institute for Information Transmission Problems,
Russian Academy of Sciences}
\email{tpanov@mech.math.msu.su}

\thanks{The research was carried out at the IITP RAS and supported by the Russian
Science Foundation (project no.~14-50-00150).}

\maketitle

We describe the cohomology of the quotient $\zk/H$ of a
moment-angle complex $\zk$ by a freely acting subtorus~$H\subset
T^m$. We establishing a ring isomorphism between $H^*(\zk/H,R)$
and an appropriate $\Tor$-algebra of the face ring $R[\sK]$, with
coefficients in an arbitrary commutative ring $R$ with unit. This
result was stated in~\cite[7.37]{bu-pa02} for a field~$R$, but the
argument was not sufficiently detailed in the case of
nontrivial~$H$ and finite characteristic. We prove the collapse of
the corresponding Eilenberg--Moore spectral sequence using the
extended functoriality of~$\Tor$ with respect to `strongly
homotopy multiplicative' maps in the
category~$\cat{dash}$~\cite{munk74}. Our collapse result does not
follow from the general results of~\cite{gu-ma74}
and~\cite{munk74}.

Let $\sK$ be a simplicial complex on $[m]=\{1,\ldots,m\}$. For
each simplex $I\in\sK$, set
\[
  (D^2,S^1)^I=\{(x_1,\ldots,x_m)\in(D^2)^m\colon x_i\in S^1=\partial
  D^2\text{ when }i\notin I\}.
\]
The \emph{moment-angle complex} is the polyhedral product
\[
  \zk=(D^2,S^1)^\sK=\bigcup\nolimits_{I\in\sK}(D^2,S^1)^I\subset(D^2)^m.
\]
$\zk$ is a manifold whenever $\sK$ is a triangulated sphere, and
can be smoothed when $\sK$ is a boundary of a polytope or is a
starshaped sphere (comes from a complete simplicial fan). Also
define
\[
  BT^\sK=(\C P^\infty,\pt)^\sK=\bigcup\nolimits_{I\in\sK}
  BT^I\subset BT^m=(\C P^\infty)^m.
\]
The cohomology of $BT^\sK$ (with coefficients in~$R$) is the
\emph{face ring} of~$\sK$:
\[
  H^*(BT^\sK)\cong R[\sK]=R[v_1,\ldots,v_m]/(v_{i_1}\cdots
  v_{i_k}\colon\{i_1,\ldots,i_k\}\notin\sK ),\quad \deg v_i=2,
\]
and there is a homotopy fibration $\zk\to BT^\sK\to BT^m$. For
more detailed background see~\cite[Ch.~4]{bu-pa15}.

The torus $T^m$ acts on $\zk$ coordinatewise and we consider
freely acting subtori $H\subset T^m$. The manifolds $\zk/H$ have
recently attracted attention as they support complex-analytic
structures, usually  non-K\"ahler, with interesting
geometry~\cite{bo-me06}, \cite{pa-us12}, \cite{ishi}.

We turn $R[\sK]$ into a module over the polynomial ring
$H^*(B(T^m/H))$ via the map $H^*(B(T^m/H))\to
H^*(BT^m)=R[v_1,\ldots,v_m]\to R[\sK]$.

\begin{theorem}
For any commutative ring $R$ with unit, there is an isomorphism of
graded algebras
\[
  H^*(\zk/H;R)\cong\Tor_{H^*(B(T^m/H);R)}(R[\sK],R).
\]
\end{theorem}
\begin{proof}
The Eilenberg--Moore spectral sequence of the homotopy fibration
$\zk/H\to BT^\sK\to B(T^m/H)$ has
$E_2=\Tor_{H^*(B(T^m/H))}(R[\sK],R)$ and converges to
$H^*(\zk/H)\cong\Tor_{C^*(B(T^m/H))}(C^*(BT^\sK),R)$. We shall
establish a multiplicative isomorphism
$\Tor_{H^*(B(T^m/H))}(R[\sK],R)\to\Tor_{C^*(B(T^m/H))}(C^*(BT^\sK),R)$;
it would also imply the collapse of the Eilenberg--Moore spectral
sequence.

For any torus $T^k$ we consider the map of $R$-modules
\[
  \varphi\colon H^*(BT^k)=(H^*(BT^1))^{\otimes k}\stackrel{i}\longrightarrow
  (C^*(BT^1))^{\otimes k}\stackrel\times\longrightarrow C^*(BT^k),
\]
where $C^*$ denotes the normalised singular cochain functor with
coefficients in~$R$, the map $i$ is the $k$-fold tensor product of
the map $H^*(BT^1)=R[v]\to C^*(BT^1)$ sending $v$ to any
representing cochain, and $\times$ is the $k$-fold cross-product.
The map $\varphi$ induces an isomorphism in cohomology.

Observe that $R[\sK]=H^*(BT^\sK)=\lim_{I\in\sK}H^*(BT^I)$ where
each $H^*(BT^I)$ is a polynomial ring on $|I|$ generators, the
(inverse) limit is taken in the category of graded algebras for
the diagram consisting of projections $H^*(BT^I)\to H^*(BT^J)$
corresponding to $J\subset I\in\sK$~\cite[3.5.1]{bu-pa15}. Now
consider the diagram
\begin{equation}\label{diagr}
\begin{array}{cccccl}
  R & \longleftarrow & H^*(B(T^m/H)) & \longrightarrow &
  \lim_{I\in\sK}H^*(BT^I) & =R[\sK]=H^*(BT^\sK)\\
  \| & & \Downarrow & & \Downarrow\\
  R & \longleftarrow & C^*(B(T^m/H)) & \longrightarrow &
  \lim_{I\in\sK}C^*(BT^I)\\
  \| & & \| & & \uparrow\\
  R & \longleftarrow & C^*(B(T^m/H)) & \longrightarrow &
  C^*(\colim_{I\in\sK}BT^I)&=C^*(BT^\sK)
\end{array}
\end{equation}
where the double arrows denote derivatives of~$\varphi$ and the
horizontal arrows on the right are induced by the maps $BT^I\to
BT^m\to BT^m/H$. All vertical arrows in~\eqref{diagr} induce
isomorphisms in cohomology (for the bottom right arrow this
follows from excision). If the diagram was commutative in the
category $\cat{da}$ of differential graded algebras (i.e.
consisted of multiplicative maps), then the standard functoriality
of $\Tor$ would have implied the required isomorphism
\[
  \Tor_{H^*(B(T^m/H))}(R[\sK],R)\cong\Tor_{C^*(B(T^m/H))}(C^*(BT^\sK),R)\cong
  H^*(\zk/H).
\]
The lower part of~\eqref{diagr} is indeed a commutative diagram
in~$\cat{da}$. The upper part is not commutative though, and the
double arrow maps are not morphisms in~$\cat{da}$ as $\varphi$ is
not multiplicative. Nevertheless $\Tor$ enjoys extended
functoriality with respect to morphisms in the
category~$\cat{dash}$, provided that the diagram~\eqref{diagr} is
homotopy commutative in~$\cat{dash}$, by~\cite[5.4]{munk74}. The
objects of $\cat{dash}$ are the same as in~$\cat{da}$, while
morphisms $A\Rightarrow A'$ are coalgebra maps $BA\to BA'$ of the
bar constructions. The map $\varphi$ and the double arrows
in~\eqref{diagr} are morphisms in~\cat{dash} by~\cite[7.3]{munk74}
(the extra condition on~$\mbox{\it Sq}_1$ is obviously satisfied
as $H^*(BT^k)$ is zero in odd degrees). To see that the upper
right square in~\eqref{diagr} is homotopy commutative, it is
enough to establish the homotopy commutativity of the diagram
\[
\begin{array}{ccccc}
  H^*(B(T^m/H)) & \longrightarrow & H^*(BT^I) & \longrightarrow & H^*(BT^J)\\
  \Downarrow & & \Downarrow & & \Downarrow\\
  C^*(B(T^m/H)) & \longrightarrow & C^*(BT^I) & \longrightarrow & C^*(BT^J)
\end{array}
\]
for any $J\subset I\in\sK$. The right square is commutative in the
standard sense by the construction of~$\varphi$ (note that we are
using normalised cochains), while the left square is homotopy
commutative by~\cite[7.3]{munk74}.

It remains to prove that the isomorphism
$\Tor_{H^*(B(T^m/H))}(R[\sK],R)\to\Tor_{C^*(B(T^m/H))}(C^*(BT^\sK),R)$
is multiplicative. We have a commutative diagram
\[
\begin{array}{ccccc}
  R\otimes R & \longleftarrow & C^*(B(T^m/H)) \otimes C^*(B(T^m/H)) & \longrightarrow &
  C^*(BT^\sK) \otimes C^*(BT^\sK)\\
  \downarrow & & \Downarrow & & \Downarrow\\
  R & \longleftarrow & C^*(B(T^m/H)) & \longrightarrow & C^*(BT^\sK)
\end{array}
\]
Using the functoriality of $\Tor$ in $\cat{dash}$ we get a natural
map
\[
  \Tor_{C^*(B(T^m/H)) \otimes C^*(B(T^m/H))}(C^*(BT^\sK) \otimes C^*(BT^\sK),R\otimes
  R)\to
  \Tor_{C^*(B(T^m/H))}(C^*(BT^\sK),R)
\]
which, composed with the classical K\"unneth-like map
\begin{multline*}
  \Tor_{C^*(B(T^m/H))}(C^*(BT^\sK),R)\otimes
  \Tor_{C^*(B(T^m/H))}(C^*(BT^\sK),R)\\\to
  \Tor_{C^*(B(T^m/H)) \otimes C^*(B(T^m/H))}(C^*(BT^\sK) \otimes C^*(BT^\sK),R\otimes
  R),
\end{multline*}
gives the multiplicative structure in
$\Tor_{C^*(B(T^m/H))}(C^*(BT^\sK),R)$. It can be checked that this
multiplicative structure is the same as the one defined via the
Eilenberg--Zilber theorem and used in the Eilenberg--Moore
isomorphism $\Tor_{C^*(B(T^m/H))}(C^*(BT^\sK),R)\cong H^*(\zk/H)$,
see~\cite[p.~46]{munk74}.

The product in $\Tor_{H^*(B(T^m/H))}(H^*(BT^\sK),R)$ is defined
similarly. Denote $B=C^*(B(T^m/H))$ and $M=C^*(BT^\sK)$. The
diagram
\[
\begin{array}{ccccc}
  \Tor_{HB}(HM,R)\otimes \Tor_{HB}(HM,R) & \longrightarrow &
  \Tor_{HB\otimes HB}(HM\otimes HM,R\otimes R) & \longrightarrow &
  \Tor_{HB}(HM,R)\\
  \downarrow & & \downarrow & & \downarrow\\
  \Tor_{B}(M,R)\otimes \Tor_{B}(M,R) & \longrightarrow &
  \Tor_{B\otimes B}(M\otimes M,R\otimes R) & \longrightarrow &
  \Tor_{B}(M,R)
\end{array}
\]
in which the vertical arrows are isomorphisms of $R$-modules, is
commutative, because the corresponding 3-dimensional diagram in
which each $\Tor_{B}(M,R)$ is replaced by $R\leftarrow B\to M$ is
homotopy commutative in $\cat{dash}$. Therefore, the $R$-module
isomorphism $\Tor_{HB}(HM,R)\to \Tor_{B}(M,R)$ is multiplicative
with respect to the multiplicative structure given.
\end{proof}

\begin{remark}
When $R$ is a field of zero characteristic, one can avoid
appealing to the category $\cat{dash}$ by using a commutative
cochain model in the argument above. One can also avoid
using~$\cat{dash}$ when $H$ is a trivial
subgroup~\cite[8.1.12]{bu-pa15}.
\end{remark}

Examples of quotients $\zk/H$ include compact toric manifolds
(when $H$ has maximal possible dimension), in which case $R[\sK]$
is a free $H^*(B(T^m/H))$-module, and Theorem~1 reduces to the
well-known description of the cohomology
(see~\cite[\S7.5]{bu-pa02}).

Another series of examples are `projective' moment-angle manifolds
$\zk/S^1_d$ corresponding to the diagonal subcircle
$H=S^1_d\subset T^m$. When $\sK$ is the boundary of a polytope,
$\zk/S^1_d$ admits a complex-analytic structure as an
\emph{LVM-manifold}~\cite{bo-me06}. In this case Theorem~1
together with the Koszul resolution gives the following
isomorphism:
\[
  H^*(\zk/S^1_d)\cong H(\Lambda[t_1,\ldots,t_{m-1}]\otimes R[\sK],d)
\]
where the cohomology of the differential graded algebra on the
right hand side is taken with respect to the differential $d
t_i=v_i-v_m$, $dv_j=0$, $\deg t_i=1$, see~\cite[7.39]{bu-pa02}.

\medskip

The author is grateful to Matthias Franz for drawing attention to
the incompleteness of the argument for~\cite[7.37]{bu-pa02} and
fruitful discussions.

\end{document}